\documentclass{amsart}
\usepackage{amsmath,amsfonts,amssymb,amscd,amsthm,epsf}
\newtheorem{prop}{Proposition}[section]
\newtheorem{thm}[prop]{Theorem}

\newtheorem{cor}[prop]{Corollary}

\newtheorem{lem}[prop]{Lemma}
\newtheorem{conj}[prop]{Conjecture}

\theoremstyle{definition}
\newtheorem{de}[prop]{Definition}

\newtheorem{examples}[prop]{Examples}
\theoremstyle{remark}

\def\zed{{\mathbb Z}}

\def\R{{\mathbb R}}
\def\RP{{\mathbb R \mathbb P}}

\def\int{\mathop{\rm int}\nolimits}

\def\id{\mathop{\rm id}\nolimits}
\def\GL{\mathop{\rm GL}\nolimits}
\def\SL{\mathop{\rm SL}\nolimits}

\def\SO{\mathop{\rm SO}\nolimits}

\def\det{\mathop{\rm det}\nolimits}
\def\tr{\mathop{\rm tr}\nolimits}

\begin{document}
\title{More Cappell-Shaneson spheres are standard}
\author{Robert E. Gompf}
\thanks{Partially supported by NSF grant DMS-0603958.}
\address{The University of Texas at Austin, Mathematics Department RLM 8.100, Attn: Robert Gompf,
2515 Speedway Stop C1200, Austin, Texas 78712-1202}
\email{gompf@math.utexas.edu}
\begin{abstract}
Akbulut has recently shown that an infinite family of Cappell-Shaneson homotopy 4-spheres is diffeomorphic to the standard 4-sphere. In the present paper, a different method shows that a strictly larger family is standard. This new approach uses no Kirby calculus except through the relatively simple 1979 paper of Akbulut and Kirby showing that the simplest example with untwisted framing is standard. Instead, hidden symmetries of the original Cappell-Shaneson construction are exploited. In the course of the proof, an example is given showing that Gluck twists can sometimes be undone using symmetries of fishtail neighborhoods.
\end{abstract}
\maketitle


\section{Introduction}

The smooth 4-dimensional Poincar\'e Conjecture, that every homotopy 4-sphere is diffeomorphic to the standard 4-sphere $S^4$, is perhaps the last of the great unsolved conjectures of classical manifold theory. It is widely believed to be false, largely because of the multitude of potential counterexamples. The best-known, and historically the most promising, family of potential counterexamples was constructed by Cappell and Shaneson in the 1970's \cite{CS}. These were indexed by an infinite family of matrices, together with a $\zed/2$ choice of a framing. The examples with the twisted choice of framing seemed the most intractable, hence, the most likely to be exotic $S^4$'s, especially since two of them were known to be double covers of exotic $\RP^4$'s \cite{CS2}. It took years of work with difficult Kirby calculus computations to show by the late 1980's that even the simplest of these was diffeomorphic to $S^4$ \cite{AK1}, \cite{AK2}, \cite{Sig0}. More Kirby calculus \cite{Sigm} yielded handle diagrams of an infinite family of such examples, published in 1991, but no further progress was made until June 2009, when Akbulut \cite{A} found a simple proof that these diagrams all represent diffeomorphic manifolds, so they are $S^4$ by previous results. In the present paper, we show that a strictly larger family is $S^4$. Our method may ultimately show that all Cappell-Shaneson homotopy spheres are $S^4$, suggesting the intriguing possibility that the smooth 4-dimensional Poincar\'e Conjecture may be true after all. Our main technique is new, exploiting hidden symmetries of the original construction of Cappell and Shaneson. That construction exhibits their examples as being precisely those homotopy 4-spheres admitting 2-knots fibered by punctured 3-tori, with monodromy given by the associated matrix. For any 2-torus in the fiber, a multiplicity-1 logarithmic transformation can also be described as a change in the monodromy, but by finding a suitable vanishing cycle, we can sometimes show that the diffeomorphism type is unchanged. This gives a new relation among the matrices, preserving the associated diffeomorphism types. No Kirby calculus is used in this paper, and the argument only depends on it through the relatively simple proof that appeared in \cite{AK1} in 1979.

As we will see, the Cappell-Shaneson homotopy 4-spheres are indexed by a $\zed/2$ framing choice and a conjugacy class of {\em  Cappell-Shaneson matrices} $A\in SL(3,\zed)$, defined by the condition $\det(A-I)=1$. The most thoroughly studied subfamily is given by the matrices
$$ A_m = \left[ \begin{array}{ccc} 0 & 1 & 0 \\ 0 & 1 & 1 \\ 1 & 0 & m+1 \end{array} \right].$$
These represent the simplest conjugacy class for each value of the trace. They form a ``cofinite'' subfamily in the sense that each trace is realized by only finitely many conjugacy classes of Cappell-Shaneson matrices, although the number of such classes appears to increase superlinearly with the trace. (See Section~\ref{CS} and \cite{AR} for further discussion.) The homotopy spheres corresponding to $m=0,4$ and the twisted framing arise as double covers of exotic $\RP^4$'s \cite{CS2}, \cite{AR}. The first progress in trivializing  Cappell-Shaneson spheres was due to Akbulut and Kirby \cite{AK1} in 1979, showing via Kirby calculus that the example with $m=0$ and untwisted framing is $S^4$. Aitchison and Rubenstein \cite{AR} observed that this framing is not the one arising (as claimed in \cite{AK1}) from an exotic $\RP^4$, and showed that for all $A_m$ the example with untwisted framing is standard. (They used different matrices, but those are easily seen to be conjugate to $A_m$ \cite{Sig0}.) The twisted framings have been much harder. Akbulut and Kirby worked for 6 years on the $m=0$ case before publishing an elegant handle diagram of it with no 3-handles and only two 1-handles \cite{AK2}. This was then shown by the author to be diffeomorphic to $S^4$ \cite{Sig0}. That same paper exhibited handle decompositions of the Cappell-Shaneson spheres corresponding to all the matrices $A_m$ and twisted framing; the derivation by a long sequence of Kirby moves appeared in \cite{Sigm}. These diagrams stood for two decades until a clever observation of Akbulut \cite{A} showed they were all diffeomorphic to the $m=0$ case, and hence standard.

The arguments trivializing the Kirby diagrams of Cappell-Shaneson spheres with twisted framing all end with a similar trick, first seen in \cite{Sig0}. In that example, an initial simplification of the diagram from \cite{AK2} (following a group-theoretic computation generated by computer and simplified by Akbulut and Casson) resulted in a diagram with two twist boxes. It soon became apparent that the numbers of twists could be varied, resulting in a family of homotopy spheres parametrized by $n\in\zed$ (not obviously related to the Cappell-Shaneson examples, except for the case of interest, $n=4$). The $n=0$ case was easily seen to be $S^4$, so the crucial step was to show that the examples were all diffeomorphic to each other. This was accomplished by introducing a 2-handle/3-handle pair, with the attaching curve of the 2-handle appearing in a strategic location (although unknotted in the $S^3$ obtained by surgery on the rest of the diagram). The 2-handle was then dragged around a torus embedded in a 3-manifold obtained by surgery on just some components, and returned to its original location. Some handle slides were needed in order to move the 2-handle past another 2-handle, changing some twists, and hence $n$, by 1. Recently, the author examined the proof in the upside-down handle picture, where it is relevant to generalizing Property R. From this viewpoint, a 1-handle/2-handle pair is introduced, but the 2-handle again slides around an embedded torus to change twists. (This appears in \cite{PropR} in the context of Property R.) Akbulut's proof again introduces a 2-3 pair and runs the 2-handle around a torus to change $m$ by 1. (In that case, one needs to strategically locate {\em two} circles that are unknotted in the surgered $S^3$ in order to find the whole torus --- the author essentially did this in \cite{Sig0}, but missed the crucial punch line!)

These observations suggest a common underlying mechanism that we will exploit directly. The crucial feature of the examples above is an embedded $D^2\times T^2$ with a 2-handle attached along an essential circle of $T^2$. The framing of the 2-handle is $-1$ in each case, so we obtain what is called a {\em fishtail neighborhood}. In the present paper, we go back to the original description of the Cappell-Shaneson manifolds and look for fishtail neighborhoods there. We then use symmetries of the fishtail neighborhoods (logarithmic transformations) to produce diffeomorphisms between examples corresponding to different conjugacy classes of matrices, with different traces. The main tool, Theorem~\ref{twist}, whose proof requires no handle diagrams, is stated in the more general context of 4-manifolds with fibered 2-knots, although it seems most useful when the fibers are punctured 3-tori. In the present paper, we apply this theorem and matrix algebra to show that for many Cappell-Shaneson matrices, the resulting pair of homotopy spheres is diffeomorphic to the pair arising from the matrix $A_0$. The latter two homotopy spheres are standard, by \cite{AK1} for the untwisted framing and by \cite{Sig0} for the twisted one. However, \cite{Sig0} relies on a long and tricky sequence of handle moves beginning with \cite{AK2}, a result of years of study. Section~\ref{Framing} is devoted to a direct proof via Theorem~\ref{twist} that the two Cappell-Shaneson spheres constructed from $A_0$ are diffeomorphic to each other. As a result, the theorems in this paper only depend on Kirby calculus through \cite{AK1}, in which one simple Cappell-Shaneson sphere (from $A_0$ with untwisted framing) is exhibited as a handle diagram that immediately cancels. Section~\ref{Framing} can also be interpreted as an example showing that a Gluck twist can sometimes be undone by exploiting symmetries of fishtail neighborhoods. The example seems somewhat special, however, involving several strategically placed fishtails. In fact, it sometimes is possible to create exotic smooth structures by Gluck twists on spheres in 4-manifolds, at least in the nonorientable case \cite{Gluck}.

The main results of this paper are given in Section~\ref{CS}. Example~\ref{Am}(a) shows that the pairs of diffeomorphism types given by the matrices $A_m$ are independent of $m$. Thus, they are all the standard $S^4$, and we recover the results of Akbulut and Aitchison-Rubenstein. Example~\ref{Am}(b) then gives a previously unknown result, exhibiting a Cappell-Shaneson matrix not conjugate to any $A_m$ and showing that the resulting homotopy spheres are standard. The subsequent theorems then give large classes of Cappell-Shaneson matrices for which the corresponding homotopy spheres must be standard. These theorems are presumably not optimal, but already seem to indicate that if any exotic Cappell-Shaneson spheres exist, their matrices must have several moderately large entries (Corollary~\ref{entries}). While conjugacy classes in $\SL(3,\zed)$ can be understood using algebraic number theory (see the appendix of \cite{AR}), it seems difficult to relate that technology to specific families of Cappell-Shaneson matrices. We provide no further results about conjugacy classes, although there are presumably infinitely many beyond those of $A_m$ covered by the theorems. The difficulty of analyzing conjugacy classes suggests that matrix-level results such as Theorems~\ref{mod d} and \ref{mod a} may be useful even for some matrices that are (nonobviously) conjugate to some $A_m$. The author conjectures that all Cappell-Shaneson spheres are covered by these or similar consequences of the main theorem.

\section{The main tool}

The Cappell-Shaneson homotopy 4-spheres arise from the following construction. Let $M$ be a connected, oriented 3-manifold with an orientation-preserving self-diffeomorphism $\varphi$. Without loss of generality, we may assume $\varphi$ restricts to the identity in a neighborhood of some point $p\in M$. Let $X_\varphi$ be the mapping torus $\R\times M/(t,x)\sim (t-1,\varphi(x))$. Then $\R\times \{ p\}$ descends to a circle $C\subset X_\varphi$ with a canonical framing. For $\epsilon=0,1$ let $X^\epsilon_\varphi$ be obtained from $X_\varphi$ by surgery on $C$ with the canonical ($\epsilon=0$) or the noncanonical ($\epsilon=1$) framing. (Thus, $X^\epsilon_\varphi$ contains a 2-sphere whose complement $X_\varphi-C$ is fibered by $M-\{p\}$.) The Cappell-Shaneson examples arise when $M$ is the 3-torus, so $\varphi$ is obtained from some $A\in SL(3,\zed)$, and $X^\epsilon_\varphi$ is a homotopy 4-sphere if and only if $\det(A-I)=\pm 1$. However, our main tool applies in the general case:

\begin{thm}\label{twist}
For $M$ and $\varphi$ as above, suppose there is an oriented circle $\alpha\subset M$ containing $p$, and a torus $T\subset M$ containing both $\alpha$ and $\varphi(\alpha)$, in which the two circles have algebraic intersection $\pm 1$ with $\alpha\cap\varphi(\alpha)$ connected. Suppose that the framings induced by $T$ on $\alpha$ and $\varphi(\alpha)$ correspond under $\varphi$. If $\delta: M\to M$ denotes the Dehn twist along $T$ parallel to $\varphi(\alpha)-\alpha$, then $X^\epsilon_{\varphi\circ\delta^k}= X^\epsilon_\varphi = X^\epsilon_{\delta^k\circ\varphi} $ for all $k\in\zed$ and $\epsilon=0,1$.
\end{thm}

\noindent A more careful definition of $\delta$ is to identify a collar of $T$ in $M$ with $I\times S^1\times S^1$ so that the first $S^1$-factor is homologous to $\varphi_*[\alpha]-[\alpha]$, then take $\delta=(\text{Dehn\ twist})\times \id_{S^1}$. Since $p$ lies on the boundary of the collar, we can assume it is outside the support of $\delta$ so that the framings on $C$ are undisturbed. Strictly speaking, the notation should also specify the side of $T$ on which the collar lies, although the theorem applies to both. However, there is no ambiguity when $M=T^3$ as in the Cappell-Shaneson setting.

The main idea of the proof is to locate a fishtail neighborhood in $X^\epsilon_\varphi$ and invoke the following lemma.

\begin{lem} \label{fishtail}
Suppose $N=D^2\times S^1\times S^1$ is embedded in a 4-manifold $X$. Suppose there is a disk $D\subset X$ intersecting $N$ precisely in $\partial D = \{q\}\times S^1$ for some $q\in \partial  D^2\times S^1$, and that the normal framing of $D$ in $X$ differs from the product framing on $\partial D\subset \partial N$ by $\pm 1$ twist. Then the diffeomorphism type of $X$ does not change if we remove $N$ and reglue it by a $k$-fold Dehn twist of $\partial N$ along $S^1\times S^1$ parallel to $\gamma=\{q\}\times S^1 $.
\end{lem}

This lemma is well-known in various forms. The resulting submanifold $\Phi=D^2\times T^2\cup _\gamma (\text{2-handle})$ of $X$ is called a {\em fishtail neighborhood} (up to orientation) and is a regular neighborhood of a sphere with a double point. The gluing operation is called a multiplicity-1 {\em logarithmic transformation} with direction $\pm\gamma$ and auxiliary multiplicity $|k|$ (or in a different context, a {\em Luttinger surgery}). The lemma is at the heart of the proof \cite{M} that simply connected elliptic surfaces with fixed $b_2$ are determined by their multiplicities. A different application appears in \cite{Fish}. We include a proof of the lemma for completeness.

\begin{proof}
We can assume (after isotopy) that the gluing diffeomorphism is supported away from $\gamma$. Thus, it extends by the identity over $\partial \Phi$. It now suffices to see that the diffeomorphism extends over $\Phi$. The general case then follows from the case $k=1$. Interpret $\partial N$ as the trivial $\partial D^2\times S^1$-bundle over the middle factor of $N$, so that $\gamma$ corresponds to the last factor of the fiber. Adding the 2-handle to $N$ changes the boundary by $\pm 1$-surgery on $\gamma$. This, in turn, can be interpreted as changing the bundle monodromy on $\partial D^2\times S^1$ to a Dehn twist $\psi$ along $\gamma$. (Think of the tubular neighborhood of $\gamma$ as an annulus in the fiber crossed with an interval in the base, then perform the surgery so that the gluing map is supported in a single fiber.) Thus, we have identified $\partial \Phi$ with the bundle $\R\times T^2/(t,x)\sim (t-1,\psi(x))$. The gluing diffeomorphism on $\partial \Phi$ specified by the lemma is a Dehn twist parallel to $\gamma$ along the torus descending from $\R \times \gamma \subset \R \times T^2$, namely $\psi$ on each fiber. To extend this over $\Phi$, work in a collar $I\times\partial \Phi$ with the first factor parametrized by $s\in [0,1]$.  Then the diffeomorphism $(s,t,x)\mapsto(s,s+t,x)$ is the required one for $s=1$ and the identity for $s=0$, so it extends as required.
\end{proof}

This last diffeomorphism of $\Phi$ can also be described as isotoping the attaching circle $\gamma$ around the torus $S^1\times S^1$ and back to its original position. Thus, it is the mechanism underlying the endgames of \cite{Sig0} and \cite{A} and the examples of \cite{PropR}. Alternatively, we can view $\partial \Phi$ as the $S^1$-bundle over $T^2$ with Euler number $\pm 1$, and again see that the given diffeomorphism on this is isotopic to the identity.

\begin{proof}[Proof of Theorem~\ref{twist}]
Identify $M$ with $\{0\}\times M \subset X_\varphi$. Then $I\times \alpha\subset \R\times M$ (via the product embedding) descends to a cylinder in $X_\varphi$ with embedded interior, containing the surgery circle $C$ determined by $p$, and whose oriented boundary is $-\alpha \cup \varphi(\alpha)\subset T \subset M$. Since $\varphi$ is the identity near $p$, the curves $\alpha$ and $\varphi(\alpha)$ intersect in an arc there, oriented compatibly. By hypothesis, they cross there and have no other intersections. Thus, the image of the cylinder is a punctured torus $F$ embedded in $X_\varphi$, whose boundary $\partial F\subset T$ is made from  $\varphi(\alpha)$ and $-\alpha$ by deleting the common segment --- in particular, it is homologous to $\varphi_*[\alpha]-[\alpha]$ in $T$. The heavy solid lines in Figure~\ref{F} show the intersection of $F$ with $I_0\times T\subset I_0\times M$, where the latter is a bicollaring of $M$ in $X_\varphi$ respecting the local product structure.

\begin{figure}
\centerline{\epsfbox{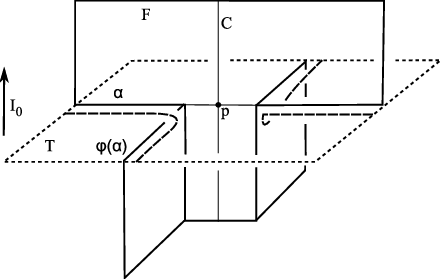}}
\caption{}\label{F}
\end{figure}

To construct $N$ in $X_\varphi$ and $\Phi$ in $X^\epsilon_\varphi$, fix a collar $[0,2]\times T^2$ of $T$ in $M$, with $T$ identified with $\{ 2\}\times T^2$. Let $N=I_0\times [0,1] \times T^2\subset I_0\times M\subset X_\varphi$. Let $F'$ be the punctured torus obtained from $F$ by connecting it to $N$ using the obvious collar $[1,2]\times \partial F$ of $\partial F$ in $M$. Then $F'$ determines a normal framing on $\partial F'\subset \partial N$, obtained by restricting any normal framing of $F'\subset X_\varphi$ (and independent of choice of the latter). The framing determined by $F'$ has $\pm 1$ twist relative to the framing induced by the inclusions $\partial F'\subset \{1\}\times T^2\subset \partial N$. This can be seen by explicitly constructing a disjoint parallel copy of $F'$ in $X_\varphi$. Since the framings induced by $T$ on $\alpha$ and $\varphi(\alpha)$ correspond under  $\varphi$, we can simply visualize this copy pushed away from $F$ toward the reader in Figure~\ref{F}, and note that its boundary is the dashed curve. The full twist of that curve about $\partial F$ is retained when we slide it to $\partial N$ along the collar $[1,2]\times \partial F$ (the fourth coordinate in the figure). Alternatively, project $F$ to $T$, where it appears (essentially) as the standard diagram for resolving a knot crossing (Figure~\ref{crossing}). After an isotopy in $I_0\times T$, the band connecting $\alpha$ and $\varphi(\alpha)$ is embedded in the plane. The parallel copy of $F$ then determines the blackboard framing, which picks up a twist when we remove the crossing by a Type I Reidemeister move. Having determined the framing induced by $F'$, we now recall that $F'$ contains the surgery curve $C$ in its interior, so we may perform the surgery pairwise to obtain a disk $D\subset X^\epsilon_\varphi$ for $\epsilon=0,1$. The framing induced by $D$ on $\partial D=\partial F'$ agrees with that induced by $F'$ since $D$ is homologous in $H_2(X^\epsilon_\varphi,N)$ to a copy of $F'$ obtained by pushing off of $C$ before the surgery. (Note that the surgery creates no 2-homology since the meridian $S^2$ to $C$ bounds a punctured copy of $M$.) Thus, $N$ and $D$ satisfy the hypotheses of Lemma~\ref{fishtail} with $\gamma$ homologous to the primitive class $\varphi_*[\alpha]-[\alpha]$ in $T$.

\begin{figure}
\centerline{\epsfbox{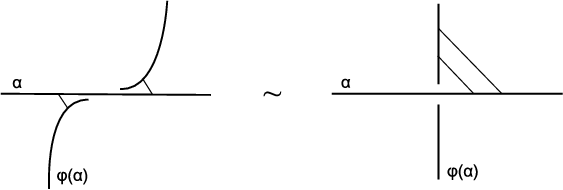}}
\caption{}\label{crossing}
\end{figure}

To complete the proof, shrink $N$ inside itself to obtain $N'\subset \int N$ with no corners. By Lemma~\ref{fishtail}, the diffeomorphism type of $X^\epsilon_\varphi$ is unchanged if we cut out $N'$ and reglue it by a $k$-fold Dehn twist along $T^2$ parallel to $\gamma$. Extend the diffeomorphism across the collar to $\partial N$, arranging its support to be in the front face $\partial_+I_0\times [0,1]\times T^2$. It follows that $X^\epsilon_\varphi$ is unchanged if we cut along this face and reglue by the given Dehn twist, or equivalently, precede the surgery by cutting along an $M$-fiber and regluing by $\delta^k$. Thus, we have produced a diffeomorphism from $X^\epsilon_\varphi$ to  $X^\epsilon_{\varphi\circ\delta^k}$. To get $X^\epsilon_{\delta^k\circ\varphi}$, apply the result for $-k$ to $\varphi^{-1}$ and then flip the sign of $t$.
\end{proof}

\section{Cappell-Shaneson spheres}\label{CS}

Cappell-Shaneson spheres were studied extensively by Aitchison and Rubenstein \cite{AR}. These examples are obtained as in the previous section with $M=T^3$ and $\varphi$ obtained from some  matrix $A\in \SL(3,\zed)$ with $\det(A-I)=\pm 1$, by straightening the corresponding linear diffeomorphism of $T^3$ to the identity near 0. Since inverting $A$ preserves $X^\epsilon_\varphi$ but flips the sign of $\det(A-I)$, we assume without loss of generality that the sign is $+1$, and refer to such $A$ as {\em Cappell-Shaneson matrices}. While the $T^3$-bundle $X_\varphi$ has many sections, they are all related by fiber-preserving diffeomorphisms, so we lose no generality surgering only on the 0-section. Each Cappell-Shaneson matrix $A$ then determines a pair of homotopy spheres, distinguished by the framing of the surgery, which in turn depends on how the diffeomorphism of $T^3$ determined by $A$ is straightened to the identity near 0. While Aitchison and Rubenstein described a canonical straightening procedure, we will only need to deal with the unordered pair of diffeomorphism types in this section. This pair only depends on the conjugacy class of $A$ in $\GL(3,\zed)$. The trace $\tr(A)\in \zed$ is an invariant of conjugacy, and for each value $n$ of the trace there are finitely many conjugacy classes of Cappell-Shaneson matrices, corresponding bijectively to the ideal class group of $\zed[\theta]$, for $\theta$ a root of the characteristic polynomial $\lambda^3-n\lambda^2+(n-1)\lambda-1$ of $A$. The matrices $A_m$ from the Introduction represent the simplest conjugacy class for each trace, in that they correspond to the identity element of each ideal class group. A table in \cite{AR} lists the number of conjugacy classes for each trace $n$ between $-9$ and 14, and PARI \cite{C} quickly computes the number even when $n$ is on the order of $10^8$. Experimentation suggests that the number of classes grows superlinearly with large $n$ and is symmetric under the reflection $n\mapsto 5-n$. However, the class is unique when $-4\le n\le 9$. See the appendix of \cite{AR} for further discussion.

The main advance of the present paper in this context is a matrix manipulation that preserves the corresponding pair of diffeomorphism types but can change the trace of the matrix. To apply Theorem~\ref{twist} to a given $A$, we must find a vector $v\in \zed^3$ for which $\zed^3/\langle v, Av\rangle$ is cyclic. Then there is a 2-torus $T\subset T^3$ containing the corresponding circles so that the hypotheses of the theorem are satisfied. (The framing hypothesis is automatically true since the framing induced by $T$ on each circle is constant in $\R^3$.) If such a $v$ exists, then we can find a basis of the form $(v,w,Av)$, so we lose no generality by conjugating $A$ to the standard form
 $$ \left[ \begin{array}{ccc} 0 & a & b \\ 0 & c & d \\ 1 & e & f \end{array} \right].$$
 Aitchison and Rubenstein showed that any Cappell-Shaneson matrix can be conjugated into this form, so such a $v$ indeed exists for every $A$. (Actually, Aitchison and Rubenstein obtained the transpose of this matrix, which is clearly an equivalent statement.) For $A$ in this form, Theorem~\ref{twist} allows us to compose $A$ with a Dehn twist $\delta$ along the torus spanned by the first and third coordinate axes, in the direction of $Av-v=[-1\ 0\ 1]^T$. That is, $\delta$ is isotopic to the linear diffeomorphism
 $$ \Delta= \left[ \begin{array}{ccc} 1 & -1 & 0 \\ 0 & 1 & 0 \\ 0 & 1 & 1 \end{array} \right].$$
 This allows us to change $A$ (in standard form) by adding any multiple of the second row to the third while subtracting the same multiple from the first, or by the conjugate operation on the second column, without changing the resulting pair of diffeomorphism types. In particular, we can change $f$ by any multiple of $d$ without changing $c$, (or change $c$ and $f$ by independent multiples of $d$), so we can change $\tr(A)$ by any multiple of $d$.

\begin{examples}\label{Am}
a) Consider the Cappell-Shaneson spheres corresponding to the matrices
$$ A_m = \left[ \begin{array}{ccc} 0 & 1 & 0 \\ 0 & 1 & 1 \\ 1 & 0 & m+1 \end{array} \right]$$
discussed in the Introduction. We can now easily see that these homotopy spheres are standard. Since $A_m$ is in standard form with $d=1$, we can change it to have any trace, say 2. The resulting matrix $$ \left[ \begin{array}{ccc} 0 & m+1 & m \\ 0 & 1 & 1 \\ 1 & -m & 1 \end{array} \right]$$
must be conjugate to $A_0$, since there is only one conjugacy class with trace 2. (In fact, this is easy to see directly, by adding $m$ times the first column to the second while subtracting $m$ times the second row from the first.) Since both homotopy spheres arising from $A_0$ are standard (by \cite{AK1} for the untwisted framing and Theorem~\ref{Sig0} below, or \cite{AK2} followed by \cite{Sig0}, in the twisted case), the result follows.

b) There are two conjugacy classes of Cappell-Shaneson matrices with trace $-5$, represented by $A_{-7}$ and
$$ \left[ \begin{array}{ccc} 0 & -5 & -8 \\ 0 & 2 & 3 \\ 1 & 0 & -7 \end{array} \right]$$
\cite{AR}. This latter matrix is in standard form with $d=3$. We can easily change its trace to 1, so the previous argument again shows that both associated homotopy spheres are standard. (For the untwisted framing, this was first shown in \cite{AR}.) This shows that our new method deals with a strictly larger class of Cappell-Shaneson examples than given in (a). Once again, we can avoid appealing to number theory (uniqueness of the conjugacy class of Cappell-Shaneson matrices with trace 1) by an explicit conjugation of the relevant trace-1 matrix:
$$ \left[ \begin{array}{ccc} -1 & -4 & 1 \\ 1 & 5 & 1 \\ 0 & 0 & -1 \end{array} \right]
\left[ \begin{array}{ccc} 0 & -9 & -14 \\ 0 & 2 & 3 \\ 1 & 4 & -1 \end{array}  \right]
\left[ \begin{array}{ccc} -5 & -4 & -9 \\ 1 & 1 & 2 \\ 0 & 0 & -1 \end{array} \right]=
\left[ \begin{array}{ccc} 0 & 1 & 0 \\ 0 & 1 & 1 \\ 1 & 0 & 0 \end{array} \right]=A_{-1}.$$
\end{examples}

The method of these examples easily generalizes to show:
\begin{thm}\label{mod d}
Suppose the Cappell-Shaneson matrix $A$ can be conjugated to standard form with $\tr(A)\equiv r$ mod $d$ for some $r\in\zed$ such that $-6\le r\le 9$ or there is only one conjugacy class with trace $r$ (e.g. $r=11$). Then both homotopy spheres associated to $A$ are diffeomorphic to $S^4$. In particular, this holds whenever $|d|<17$.
\end{thm}

\begin{proof}
As above, we can assume $\tr(A)=r$. For each of the listed values of $r$ except $-5$, there is only one conjugacy class with trace $r$ (see \cite{AR} for $r=-6,11$), so $A$ is conjugate to $A_{r-2}$, and the result follows from Example~\ref{Am}(a). The remaining case is settled by (b) above. The last sentence follows once we rule out the case $d=0$ by the following lemma.
\end{proof}

\begin{lem}
For any Cappell-Shaneson matrix in standard form, $d$ and either $a$ or $e$ are odd.
\end{lem}

\begin{proof}
If this fails, so either $d$ or both $a$ and $e$ are even, then the condition $\det (A)=1=ad-bc$ guarantees that $b$ and $c$ are  both odd. But the condition $\det(A-I)=1=ad-bc$, or $b=(c-1)(f-1)-de$, contradicts this.
\end{proof}

A bit more computation reveals the following:

\begin{thm}\label{mod a}
Suppose the Cappell-Shaneson matrix $A$ can be conjugated to standard form with $c\equiv r$ mod $d$  or mod $a+ce$ where $-3\le r\le 4$. Then both homotopy spheres associated to $A$ are diffeomorphic to $S^4$. The same holds if $f\equiv 1$ mod $d$ or $a+ce$.
\end{thm}

\begin{cor}\label{entries}
If a Cappell-Shaneson homotopy sphere is not diffeomorphic to $S^4$, then its matrix, in standard form, must have $|d|\ge17$, $|a+ce|\ge9$, $|c-\frac12|>4$ and $|c+f-\frac32|>8$. $\qed$
\end{cor}

\noindent The matrix can always be conjugated to standard form with $e=0$. (See the proof below, which also implicitly shows $a+ce$ is odd.) However, this standard form is far from unique, so the corollary and preceding theorems are more powerful than they initially appear. We discuss this further below. Theorem~\ref{mod a} can be proved without appealing to algebraic number theory (unlike Theorem~\ref{mod d} which used uniqueness of conjugacy classes with small trace). Thus, we can eliminate number theory from the proof of Corollary~\ref{entries} by weakening the first inequality to $|d|\ge9$ and deleting the last.

\begin{proof} [Proof of Theorem~\ref{mod a}]
We manipulate $A$, ignoring its third column. In addition to our new move, we use elementary conjugations adding $k$ times row (a) to row (b), then subtracting $k$ times column (b) from column (a). (Recall that row operations commute with column operations, since they are matrix multiplications on opposite sides.) It suffices to assume $e=0$, after subtracting $e$ times the first column from the second, and the corresponding row operation that replaces $a$ by $a+ce$. Without disturbing standard form with $e=0$, we can change the pair $(a,c)$ to either (i)~$(a, c+ka)$ or  (ii)~$(a+kc(c-1),c)$ for any $k\in\zed$. For (i) we conjugate, adding $k$ times the first row to the second, with the corresponding column operation. The latter changes the first two entries of the first column, but these can be reset to 0 by row operations whose corresponding column operation only affects the third column. For (ii), use the new move to subtract $k$ times the second row from the first, replacing $a$ by $a-kc$ and $e=0$ by $kc$. Resetting $e$ to 0 as before completes the process. Since $a$ is odd, both (i) and (ii) are nontrivial operations, except in the case $c(c-1)=0$. But this case is easy, since $c=1$ implies $b=(c-1)(f-1)=0$, and $b$ or $c=0$ implies $ad=ad-bc=1$, so $d=\pm 1$. We can now invoke Theorem~\ref{mod d} or reduce to some $A_m$ by hand. (In fact, when $c=1$, $A$ is already forced to be $A_{f-1}$, after we possibly reverse the signs of the middle row and column.) The last sentence of the theorem follows immediately: We can assume $f=1$, invoking (i) and conjugation-invariance of the trace if necessary, and recalling that we replaced $a$ by $a+ce$ when reducing to the $e=0$ case. As before, $b=0$.

To complete the proof, it suffices to assume $c=r$, by arranging this beforehand (in the mod $d$ case) or applying (i). Thus, $-3\le c\le 4$, so $0\le c(c-1)\le 12$. Applying (ii), we can now arrange $-5\le a\le5$ (since $a$ is odd). By (i) again, we can now assume $|c|\le2$, so $0\le c(c-1)\le 6$. Continuing to alternately apply (ii) and (i), we reduce to the case $c=0$ that we have already finished.
\end{proof}

These theorems suggest the following conjecture, which implies that all Cappell-Shaneson homotopy 4-spheres are standard.

\begin{conj}
All Cappell-Shaneson matrices are equivalent under the relation generated by conjugacy and multiplying standard forms by $\Delta$.
\end{conj}

\noindent We have already shown that all such matrices satisfying the hypotheses of Theorem~\ref{mod d} or \ref{mod a} are equivalent to $A_0$, primarily by using $\Delta$ and elementary conjugations preserving the lower left 1 of standard form. One can hope to obtain much more by using more general conjugations. For example, adding $a$ times the third column minus $b$ times the second to the first (with corresponding row operations) destroys the lower left 1 but creates a new 1 in the middle of the first column. This can be moved to the lower left by interchanging the second and third rows (and columns). Restoring standard form gives a new matrix with little resemblance to the original. (If $e=0$, we have replaced $d$ by $-a^2-(1+af)[c+2ab-f+(d+b^2)(1+af)]$ and $a+ce$ by $b+(d+b^2)[f-ab-(d+b^2)(1+af)]$, with the new $c$ enclosed by the last square brackets.) Presumably, this move makes many new matrices accessible by the theorems and corollary, although it is not clear how to proceed systematically. Expanding on \cite{AR}, we can at least reduce the number of parameters needed to express standard-form Cappell-Shaneson matrices. Note that $b=(c-1)(f-1)-de$ is determined by the other entries. Eliminating it from the equation $ad=bc+1$, we get the factorization $(a+ce)d=c(c-1)(f-1)+1=-p(c)$, where $p(\lambda)=\lambda^3-(c+f)\lambda^2+(c+f-1)\lambda-1$ is the characteristic polynomial for Cappell-Shaneson matrices with trace $c+f$. We conclude that standard-form Cappell-Shaneson matrices correspond bijectively to triples $c,e,f\in\zed$ and factorizations of $-p(c)$ (although $e$ can be set to 0 by conjugation). Corollary~\ref{entries} now gives the additional constraint that a Cappell-Shaneson matrix is equivalent to $A_0$ unless $|p(c)|\in\zed$ splits into two factors, $\ge 9$ and 17, respectively.

\section{Changing the framing}\label{Framing}

So far, we have only considered the unoriented pairs of homotopy spheres associated to a given Cappell-Shaneson matrix. We now distinguish within each pair. By repeatedly applying Theorem~\ref{twist}, we construct a diffeomorphism between the two homotopy 4-spheres associated to the matrix $A_0$. All of our results showing that Cappell-Shaneson homotopy spheres are standard then depend on Kirby calculus only through the original paper \cite{AK1}, which is far easier than \cite{AK2} followed by \cite{Sig0}. Another way of viewing this section is that Gluck twists on fibered 2-knots can sometimes be undone by repeated application of Theorem~\ref{twist}.

Recall that the setup for Theorem~\ref{twist} requires the monodromy $\varphi$ to fix a neighborhood of some point $p$ in the fiber $M$. In the case of Cappell-Shaneson spheres, we use linear diffeomorphisms of $T^3$, with fixed point $p=0$ at which the tangent space is not fixed. To apply the theorem, we must isotope the diffeomorphism rel 0 to fix a neighborhood of 0.

\begin{de}
For a 3-dimensional, real vector space $V$, a {\em straightening} of a linear transformation $A\in \GL(V)$ is a homotopy class of paths in $\GL(V)$ from $A$ to the identity $I$.
\end{de}

\noindent When $\det(A)>0$ there are exactly two straightenings of $A$, differing by an element of $\pi_1(\GL(V))=\zed/2$. If $\varphi_0:M\to M$ is a diffeomorphism fixing $p$, then for any straightening $\sigma$ of the derivative $d(\varphi_0)_p$, there is an isotopy $\varphi_t$ rel $p$ for which (i) $\varphi_1$ fixes a neighborhood of $p$ and (ii) $d(\varphi_t)_p$ represents $\sigma$. The diffeomorphism $\varphi_1$ is uniquely determined by $\varphi_0$ and $\sigma$, up to isotopy rel a neighborhood of $p$, and characterized by the existence of an isotopy from $\varphi_0$ satisfying (i) and (ii). (Given another such isotopy $\varphi'_t$, we can glue it to $\varphi_t$ to obtain an isotopy from $\varphi_1$ to $\varphi'_1$ for which the derivative at $p$ is a nullhomotopic path in $(\GL(V),I)$. Straightening the 1-parameter family of diffeomorphisms near $p$ gives the required isotopy.) Thus, the two straightenings  $\sigma$ of $d(\varphi_0)_p$  canonically determine the two possible isotopy classes (rel a neighborhood of $p$) of monodromies for the Cappell-Shaneson construction on $\varphi_0$, and hence the two resulting diffeomorphism types, which we denote by $X^\sigma_{\varphi_0}$. (For each $\sigma$ we straighten and then surger with the untwisted framing. Changing the framing is equivalent to changing $\sigma$.) For $A,B\in \GL(V)$, a homotopy class $\tau$ of paths in $\GL(V)$ from $B$ to $A$ determines a bijection from straightenings $\sigma$ of $A$ to straightenings $\tau \cdot \sigma$ of $B$ by path concatenation. Any isotopy rel $p$ between diffeomorphisms $\varphi$ and $\psi$ of $(M,p)$ determines such a $\tau$ from $d\psi_p$ to $d\varphi_p$, and $X^{\tau\cdot\sigma}_{\psi}=X^\sigma_{\varphi}$. (Glue the two isotopies and apply the above characterization.)

As our main example, let $A$ be a Cappell-Shaneson matrix, and suppose $B$ is obtained from $A$ as in Theorem~\ref{twist}. That is, after a change of basis, $B$ is obtained from $A$ by left- or right-multiplying it by a power $\Delta^k$ of the matrix $\Delta$ given in Section~\ref{CS}. The linear path from $\Delta^k$ to $I$ clearly lies in $\SL(3,\R)$. Multiplying this by $A$ gives the linear path $\tau$ from $B$ to $A$, showing that the latter path lies in $\SL(3,\R)$, and linearity is preserved when we undo the change of basis. But $B$ is also a Cappell-Shaneson matrix, and in fact $X^{\tau\cdot\sigma}_B=X^\sigma_A$ for each straightening $\sigma$. To see this, first note that Theorem~\ref{twist} applies, since the isotopy corresponding to each straightening of $A$ can be chosen to keep $\varphi(\alpha)$ within the torus $T$ of that theorem, satisfying the required hypotheses. (The isotopies will differ by a full twist in the normal bundle of the curve.) Recall that the diffeomorphism $\delta^k$ of Theorem~\ref{twist}, while isotopic to $\Delta^k$ rel 0, is actually supported away from 0 in a neighborhood of a torus. The isotopy can be taken to be linear in $T^3=\R^3/\zed^3$, and in particular its derivative is the linear straightening of $\Delta^k$ at the tangent space at 0. Theorem~\ref{twist} identifies $X^\sigma_A$ with $X^\sigma_{\delta^k \circ A}$ or $X^\sigma_{A \circ \delta^k}$ (both of which are well-defined since $\delta$ is supported away from 0). Our isotopy from $\delta^k$ to $\Delta^k$ changes the bundle monodromy to $B$ while changing the straightening to $\tau\cdot\sigma$ (for the linear $\tau$ discussed above), so $X^{\tau\cdot\sigma}_B=X^\sigma_A$ by the previous paragraph.

While Aitchison and Rubenstein \cite{AR} gave a procedure for straightening any Cappell-Shaneson matrix, it suffices here to examine their simplest special case, when the straight line path lies in $\GL(3,\R)$.

\begin{prop} \label{straight}
Let $A$ be a {\em real} Cappell-Shaneson matrix, that is, $A\in \SL(3,\R)$ with $\det(A-I)=1$. If $\tr(A)\ge 0$ then $A$ can be linearly straightened.
\end{prop}

\begin{proof}
We must show that the linear path from $A$ to $I$ lies in $\GL(3,\R)$, i.e., for $0< t < 1$ we have $0\ne\det(tA+(1-t)I) = t^3 \det(A+(\frac1t -1)I)$. Equivalently, we must show that there are no negative roots of the characteristic polynomial $-\det(A-\lambda I)=\lambda^3-\tr(A)\lambda^2+(\tr(A)-1)\lambda-1$ (where the last expression comes from setting $\lambda=0,1$ and comparing with the Cappell-Shaneson conditions). This, in turn, is the same as ruling out positive roots of $s^3+\tr(A)s^2+(\tr(A)-1)s+1$, which is obvious for $\tr(A)\ge 1$ and not much harder for $\tr(A)\ge 0$ (although we only need the case $\tr(A)=4$).
\end{proof}

\noindent The proposition clearly fails whenever $\tr(A)\le-1$; simply set $s=\frac12$.

Although any change in monodromy arising from Theorem~\ref{twist} can be realized by left or right multiplication by some $\Delta^k$ after a suitable change of basis, it will be convenient to have another example expressed without the basis change. Suppose the Cappell-Shaneson matrix $A$ has second column given by $[1\ -1 \ 0]^T$. Then the second basis vector $e_2$ and $Ae_2$ span the integer lattice in the plane perpendicular to the third axis. Thus, Theorem~\ref{twist} gives us a diffeomorphism $\delta_0$ that is a Dehn twist along a torus perpendicular the third axis, in the direction of $Ae_2-e_2= [1\ -2 \ 0]^T$. The corresponding linear diffeomorphism is
$$ \Delta_0 = \left[ \begin{array}{ccc} 1 & 0 & 1 \\ 0 & 1 & -2 \\ 0 & 0 & 1 \end{array} \right].$$
In particular, when $A$ has the required second column, we have $X^{\tau\cdot\sigma}_B=X^\sigma_A$, where $B$ is obtained from $A$ by multiplying by some $\Delta_0^k$ on the left or right, $\tau$ is the linear path from $B$ to $A$, and $\sigma$ is either straightening.

We can now prove the main theorem of the section.

\begin{thm}\label{Sig0} \cite{AK1}, \cite{AK2}, \cite{Sig0}.
The two Cappell-Shaneson spheres given by the matrix $A_0$ of Example~\ref{Am}(a) are diffeomorphic.
\end{thm}

\noindent Of course, this follows from knowing that both manifolds are $S^4$ (\cite{AK1} for the linear straightening of $A_0$ and \cite{AK2}, \cite{Sig0} for the twisted straightening), but the point is to bypass \cite{AK2}, \cite{Sig0} to conclude that many Cappell-Shaneson spheres are standard using only the relatively simple Kirby calculus argument of \cite{AK1}.

\begin{proof}
Starting with the Cappell-Shaneson matrix $A$ given below, apply  Theorem~\ref{twist} four times as shown:
$$A=\left[ \begin{array}{ccc} 0 & -1 & -2 \\ 0 & -1 & -3 \\ 1 & 2 & 5 \end{array} \right] \mapsto
 \left[ \begin{array}{ccc} 0 & 1 & 4 \\ 0 & -1 & -3 \\ 1 & 0 & -1 \end{array} \right] \mapsto
\left[ \begin{array}{ccc} 0 & 1 & 0 \\ 0 & -1 & 1 \\ 1 & 0 & 1 \end{array} \right] \mapsto
\left[ \begin{array}{ccc} 0 & 1 & 0 \\ 0 & 1 & 1 \\ 1 & 0 & 1 \end{array} \right] \mapsto
\left[ \begin{array}{ccc} 0 & -1 & -2 \\ 0 & 1 & 1 \\ 1 & 2 & 3 \end{array} \right]=B.$$
The respective moves are left multiplication by $\Delta^2$, right multiplication by $\Delta_0^2$ (note that the second column has the required form), right multiplication by $\Delta^2$, and left multiplication by $\Delta^2$. The final matrix $B$ is immediately preceded by $A_0$, so it now suffices to show that the two manifolds $X^\sigma_B$ are diffeomorphic. Since each move comes from Theorem~\ref{twist}, our previous discussion gives a diffeomorphism from $X^\sigma_A$ to $X^{\tau \cdot\sigma}_B$ for each straightening $\sigma$ of $A$, where $\tau$ is the concatenation of the linear paths between consecutive matrices above. On the other hand, $A$ and $B$ are conjugate since $\tr(A)=\tr(B)=4$. In fact, $B=CAC^{-1}$ where
$$C=\left[ \begin{array}{ccc} 2 & 1 & 2 \\ 0 & -1 & -1 \\ -1 & 0 & -1 \end{array} \right],\ \ C^{-1}=\left[ \begin{array}{ccc} 1 & 1 & 1 \\ 1 & 0 & 2 \\ -1 & -1 & -2 \end{array} \right].$$
Let $\sigma_A$ and $\sigma_B$ denote the respective linear straightenings of $A$ and $B$ (which exist since $\tr(A),\tr(B)\ge0$). Since conjugation preserves linear straightenings, we obtain a diffeomorphism  from $X^{\sigma_A}_A$ to $X^{\sigma_B}_B$. Thus, $X^{\sigma_B}_B$ is diffeomorphic to $X^{\tau \cdot\sigma_A}_B$, and it suffices to show that the 1-cycle $\tau \cdot\sigma_A\cdot\sigma_B^{-1}$ (where the last path is inverted, not the individual matrices) is nontrivial in $H_1(\GL^+(3,\R))=\pi_1(\GL(3,\R))=\zed/2$.

To analyze this 1-cycle, consider the linear path $\rho$ between $A$ and $B$, given by the family of matrices $B_t=tB+(1-t)A$, $0\le t\le 1$. These matrices all lie in $\GL^+(3,\R)$ and can be linearly straightened. This follows as in the proof of Proposition~\ref{straight}, by verifying that for $s\ge0$, $0\ne\det(B_t+sI)=s^3+4s^2+[4t(1-t)+3]s+1$, where the last expression arises by direct calculation and is clearly positive for $s\ge0$, $0\le t \le 1$. The linear straightenings comprise a 2-simplex in $\GL^+(3,\R)$ with edges $\sigma_A$, $\sigma_B$ and $\rho$. Thus, the desired 1-cycle $\tau \cdot\sigma_A\cdot\sigma_B^{-1}$ is homologous in $\GL^+(3,\R)$ to $\tau \cdot\rho$, which is obtained by linearly connecting the above five matrices into a 1-cycle in the given cyclic order. To compute its homology class, we deformation retract $\GL^+(3,\R)\to\SO(3)$ by the Gram-Schmidt procedure. Each matrix in $\tau \cdot\rho$ has first column $[0\ 0\ 1]^T$, which is unchanged by the Gram-Schmidt procedure. The second column $[a\ c\ e]^T$ becomes $[a\ c\ 0]^T$ up to positive scale, and the third column can be ignored since it is uniquely determined by the first two. It now suffices to compute the mod 2 winding number of $(a,c)$ in $\R^2$, but this is nonzero by inspection.
\end{proof}


\end{document}